\documentclass[11pt]{amsart}
\usepackage{latexsym,amssymb,amsmath,hyperref,empheq}
\usepackage{amsthm, dsfont}
\usepackage{tikz-cd}
\usepackage[all]{xy}
\usepackage[short,nodayofweek]{datetime}
\usepackage[active]{srcltx}

\leftmargin  \leftmargini
\def\NN{{\mathbb N}}

\def\RR{{\mathbb R}}
\def\CC{{\mathbb C}}

\def\P{{\mathcal P}}

\def\X{{\mathcal X}}

\def\gg{{\mathfrak g}}

\def\sign{{\mathop{sign}}}

\newcommand{\vol}{\operatorname{vol}}

\def\id{{\rm id}}

\DeclareMathOperator{\GL}{GL}

\DeclareMathOperator{\trace}{tr}
\DeclareMathOperator{\im}{Im}
\DeclareMathOperator{\re}{Re}
\DeclareMathOperator{\Sp}{Sp}

\DeclareMathOperator{\SO}{SO}

\theoremstyle{plain}
\newtheorem{thm}{Theorem}[section]

\newtheorem{lem}[thm]{Lemma}

\newtheorem{prop}[thm]{Proposition}
\newtheorem{conj}[thm]{Conjecture}

\theoremstyle{definition}

\newtheorem{exmp}[thm]{Example}

\theoremstyle{remark}

\newtheoremstyle{Acknowledgements}
  {}
    {}
     {}
     {}
    {\bfseries}
    {}
     {.5em}
     {\thmname{#1}\thmnumber{ }\thmnote{ (#3)}}
\theoremstyle{Acknowledgements}

\date{\today, \currenttime} 

\begin{document}
\title[Sturm vector valued]{Sturm's operator acting on vector valued $K$-types}

\author{Kathrin Maurischat}
\address{\rm {\bf Kathrin Maurischat}, Mathematisches Institut,
   Universit\"at Heidelberg, Im Neuenheimer Feld 205, 69120 Heidelberg, Germany }
\curraddr{}
\email{\sf maurischat@mathi.uni-heidelberg.de}

\begin{abstract}
We define Sturm's operator for vector valued Siegel modular forms obtaining an explicit description of their holomorphic projection in case of large absolute weight. However, for small absolute weight, Sturm's operator produces phantom terms in addition. This confirms our earlier results for scalar Siegel modular forms.
\end{abstract}

\maketitle
\setcounter{tocdepth}{2}
\tableofcontents

\section{Introduction}
Let $G$ be the symplectic group  of rank $m$.
Sturm's operator $\mathop{St}_\kappa$ is defined on (non-holomorphic) symplectic modular forms $f$ of weight $\kappa$ for a discrete subgroup $\Gamma\subset G$ by an integral operator on the coefficients  of the Fourier expansion $f(Z)=\sum_{T=T'}a(T,Y)e^{2\pi i\trace(TY)}$ for positive definite $T$
\begin{equation*}
 a(T,Y)\mapsto b(T)\:=\:c(\kappa)^{-1}\int_{Y>0}a(T,Y)\det(TY)^{\kappa-\frac{m+1}{2}}e^{-2\pi \trace(TY)}dY_{inv}
\end{equation*}
It is well-defined for scalar weight $\kappa>m-1$. Here $c(\kappa)$ is a constant depending only on weight and rank.
The Fourier series $\mathop{St_\kappa}(f)(Z)=\sum_{T>0} b(T)e^{2\pi i\trace(TZ)}$ allows an interpretation as holomorphic cusp form $\mathop{St_\kappa}(f)\in[\Gamma,\kappa]_0$, and indeed is the holomorphic projection $\mathop{pr}_{hol}(f)$ of $f$ in case the weight $\kappa$ is large, i.e. greater than twice the rank of the symplectic group.
This result by Sturm~\cite{sturm1}, \cite{sturm2}, and Panchishkin~\cite{panchishkin} relies on a generating system of Poincar\'e series $p_T\in[\Gamma,\kappa]_0$ for which the coefficients $b(T)$ are essentially given by the scalar product $\langle p_T,f\rangle= b(T)$.
The same result holds true for weight $\kappa=2m$ in case $m\leq 2$ (\cite{gross-zagier}, \cite{holproj}). However, in case  of weight $\kappa=3$ and rank $m=2$   we showed jointly with R.~Weissauer (\cite{phantom}) that Sturm's operator produces, along with the holomorphic projection, a second term $\mathop{ph}(f)\in [\Gamma,\kappa]_0$
\begin{equation*}
 \mathop{St}\nolimits_\kappa (f)\:=\:\mathop{pr}\nolimits_{hol}(f)+\mathop{ph}\nolimits(f)\:.
\end{equation*}
This phantom term $\mathop{ph}(f)=\mathop{St}_\kappa(\Delta^{[m]}_+(h))$ arises as the non-holomorphic Maass shift of a holomorphic  form $h\in[\Gamma,\kappa-2]$ of weight one (see section~\ref{sec-sturm} for the exact definition of $\Delta_+^{[m]}$. Later  (\cite{scalar-weight}) we generalized this result to general rank $m>2$ and $\kappa=m+1$.
However, the phenomenon of arising phantom terms in case of small weight is rather non-understood.

Therefore, here we study the case of vector valued Siegel modular forms with values in the space $V_\rho$ of an  irreducible rational representation $\rho$ of $\GL(m,\CC)$. These modular forms for example play an important role for singular weights~\cite{freitag2}.
\bigskip 

Consider the operator valued Poincar\'e series on the Siegel upper halfspace $\mathcal H$
\begin{equation}\label{def_poincare_vector_valued}
 p_T(Z)\:=\:\sum_{\gamma\in\Gamma_\infty\backslash \Gamma} \rho(J(\gamma,Z))^{-1}e^{2\pi i\trace(T\gamma\cdot Z)}\:.
\end{equation}
Here for a matrix  
$g=\left(\begin{smallmatrix}\ast&\ast\\C&D\end{smallmatrix}\right)\in G$
and $Z\in\mathcal H$ we use the $J$-factor $J(g,Z)=CZ+D$.
 We may  evaluate each single summand of these Poincar\'e series at special vectors $v\in V_\rho$ to get vector valued series. 
Candidates for $v$ are the highest weight vector $v_\rho$ or (if it exists) 
the spherical vector $v_K$.
Because of the cocycle relation $J(\tilde\gamma \gamma,Z)=J(\tilde \gamma, \gamma Z)J( \gamma,Z)$ valid for all $\gamma,\tilde \gamma \in G$,
the series $P$ has the transformation property
\begin{equation*}
 \rho(J(\tilde \gamma,Z))^{-1} p_T(\tilde \gamma Z)\:=\: p_T(Z)\:.
\end{equation*}
Assuming good convergency properties by proposition~\ref{H-translation}, $p_T(Z)v\in [\Gamma,\rho]_0$ is a vector valued holomorphic cusp form with values in $\mathop{End}(V_\rho)$.
Notice that it doesn't transform by $\mathop{Ad} \rho$, which would be more natural, but  is not compatible with its interpretation as an operator  on $V_\rho$.

For a valued non-holomorphic modular form of weight $\rho$ with Fourier expansion
\begin{equation*}
 f(Z)\:=\:\sum_{T=T'}\rho(T^{\frac{1}{2}})a(T,T^\frac{1}{2}YT^\frac{1}{2})\cdot e^{2\pi i\trace(TX)}
\end{equation*}
we define Sturm's operator by
\begin{equation*}
 \mathop{St}\nolimits_\rho(f)(Z)\:=\:\sum_{T>0}\rho(T^\frac{1}{2})b(T)e^{2\pi i\trace(TZ)}\:,
\end{equation*}
where the coefficients $b(T)$ are defined by the integral
\begin{equation*}
b(T)'\:=\: \det(T)^{-\frac{m+1}{2}}\int_{Y>0}a(T,Y)'\rho(T^\frac{1}{2})C(\rho)^{-1}\rho(Y)\rho(T^{-\frac{1}{2}}) \frac{e^{-2\pi\trace(Y)}dY_{inv}}{\det(Y)^{\frac{m+1}{2}}}\:.
\end{equation*}
Here $C(\rho)$ is an operator such that on holomorphic cuspforms $f$ Sturm's operator is the identity.
In contrast to the constant $c(\kappa)$ the scalar valued case, $C(\rho)$ must be placed carefully into the integral. In general it is known that the vector valued $\Gamma$-integrals converge in case the absolute weight of $\rho$ is large enough (\cite{godement}). But it is not clear a priori that the operators are surjective outside a discrete set of zeros and poles. Theoretically, the integrals are computable by using the Littlewood-Richardson rule once the $\Gamma$-function for all tensor powers $\mathop{st}^{\otimes n}$ of the standard representation is known. But the latter involves non-trivial combinatorics. We devote the second part of the paper to obtain some partial results.
We determine the $\Gamma$-integrals for alternating powers of the standard representation in section~\ref{sec-alt}. Further we obtain all $\Gamma$-functions for algebraic representations of $\GL(2,\CC)$ by section~\ref{sec-Gamma-rank-2}. We include some remarks on Weyl's character formula for $\Gamma$-functions in section~\ref{sec-weyl}.

We say an  irreducible representation $\rho$ of $\GL(m,\CC)$ with dominant highest weight $l=(l_1,\dots,l_n)$, where $l_1\geq l_2\geq\dots\geq l_n$, has absolute weight  $\kappa=l_n$. Like in the scalar weight case, for large absolute weight we obtain holomorphic projection by Sturm's operator:
\begin{thm}\label{theorem_hol_proj}
 Let $\rho$ be an irreducible representation of $\GL(m,\CC)$ of large absolute weight $\kappa>2m$. Assume $C(\rho)$ is an isomorphism.
 Then Sturm's operator realizes the holomorphic projection operator.
\end{thm}
Whereas, again for small absolute weight $\kappa=m+1$ this is no longer  true, as we see by the following special case.
\begin{thm}\label{thm-phantom}
 For rank $m=2$ let $\tau$ be the irreducible representation of $\GL(2,\CC)$ of highest weight $(k+1,k)$ with $k\geq 1$. Let $h\in[\Gamma,\tau]_0$ be a non-zero vector valued holomorphic cusp form of weight $\tau$. Then the image of its Maass shift $\Delta_+^{[m]}(h)$ under Sturm's operator
 \begin{equation*}
  \mathop{St}\nolimits_{\tau\otimes \det\nolimits^2}\bigl(\Delta_+^{[m]}(h)\bigr)
 \end{equation*}
is non-zero if and only if $k=1$. In particular, in case of highest weight $(4,3)$ Sturm's operator $\mathop{St}_\rho$ does not realize holomorphic projection but produces phantom terms.
\end{thm}
Our results obtained so far are limited by the explicit computability of phantom terms. Nevertheless, by \cite{phantom}, \cite{scalar-weight}, and the above, the following interpretation is at hand.
A holomorphic cusp form of weight $\rho$ generates a holomorphic representation of the symplectic group $G$ of minimal $K$-type $\rho$. In case of absolute weight $\kappa\geq m+1$ this is a (limit of) discrete series representation.
Within the root lattice of $\mathfrak{sp}_m$ and for the consistent choice of positive roots $e_1-e_2,\dots,e_{m-1}-e_m,2e_m$, 
those belong to the cone given by the $\delta$-translate of the positive Weyl chamber. 
More precisely, a representation of minimal $K$-type of highest weight $(l_1,\dots, l_m)$ is situated by its Harish-Chandra parameter $(l_1-1,l_2-2,\dots,l_m-m)$. 
Here $\delta=(m,m-1,\dots,1)$ is half the sum of positive roots.
Whereas there are some holomorphic representations outside this cone, for example those generated by $h\in[\Gamma,1]_0$.
The wall orthogonal to all short simple roots is given by $(r-1,r-2,\dots,r-m)$ for $r\geq m+1$. Here,  \cite{scalar-weight} suggests that Sturm's operator  realizes the holomorphic projection operator as long as $r>m+1$, i.e. apart from the the apex $\delta$ of the cone belonging to the minimal $K$-type $(m+1,\dots,m+1)$.
In the case of rank two theorem~\ref{thm-phantom} shows that Sturm's operator fails on the wall of the cone perpendicular to the long root. This suggests the following expectation in general.
\begin{conj}
 Sturm's operator produces phantom terms on all the facets of the cone not perpendicular to each of the short simple roots. 
 The phantom terms arise as Maass shifts of holomorphic cusp forms of small absolute weight.
\end{conj}

The paper is organized as follows. In section~\ref{Poincare-series} we study non-holomorphic Poincar\'e series as functions on the symplectic group. This is the natural point of view with respect to the Lie algebra action.
Section~\ref{sec-on-H} is devoted to the interplay of functions on group level and on the Siegel half space. We define the vector valued version of Sturm's operator, and prove its coincidence with the holomorphic projection in case of large weight. In section~\ref{sec-sturm} we show the occurrence of phantom terms.
In section~\ref{section_gamma_integral} we determine the vector valued gamma functions $\Gamma(\rho)$ as described above.

\section{Poincar\'e series}\label{Poincare-series}
\subsection{Definition and convergency}

For  the irreducible algebraic representation $(\rho,V_\rho)$ we assume  $V_\rho=\CC^N$,  $\rho: \GL(m,\CC)\to\GL(N,\CC)$, to have the properties  $\rho(x)'=\rho(x')$ and $\rho(\bar x)=\overline{\rho(x)}$ for all $x\in \GL(m,\CC)$.
This determines  $\rho$  uniquely. Here $x'$ denotes the transpose of the matrix $x$.
Then $v'\cdot\bar w$ defines the intrinsic scalar product on $V_\rho$ which is $\rho(U(m))$-invariant.

\begin{prop}\label{konvergenz_poincare}
Let $\rho$  be the irreducible rational  representation of $\GL(m,\CC)$ of dominant highest weight $(l_1,l_2,\dots,l_m)$. Let $\kappa=l_m$ be its absolute weight.
  Define
 the non-holomorphic Poincar\'e series
 \begin{equation*}
  P_T(g,s_1,s_2)\:=\:\sum_{\gamma\in\Gamma_\infty\backslash \Gamma} \rho(J(\gamma g,i))^{-1}\trace(T\im(\gamma g\cdot i))^{s_1}\det(\im(\gamma g\cdot i))^{s_2}e^{2\pi i\trace(T\gamma g\cdot i)}\:.
 \end{equation*}
Applied to any vector $v\in V_\rho$ the Poincar\'e series converge absolutely and uniformly on compact sets in the sense that this holds for 
$\lvert\!\lvert P_T(g,s_1,s_2)v\rvert\!\rvert$ in the domain
\begin{equation*}
 \left\{(s_1,s_2)\in\CC^2\mid \re s_2>m-\frac{\kappa}{2} \textrm{ and } \re(ms_2+s_1)>m^2-\frac{\sum_jl_j}{2}\right\}\:.
\end{equation*}
For fixed such $(s_1,s_2)$ the function $\lvert\!\lvert P_T(g,s_1,s_2)v\rvert\!\rvert$ is bounded and belongs to $L^2(\Gamma\backslash G)$.
In particular, in case the absolute weight $\kappa>2m$ is large, at the critical point $(s_1,s_2)=(0,0)$ the Poincar\'e series converge absolutely. 
\end{prop}
The most natural definition of Poincare series on $G$ would be one in $m$  complex variables,
\begin{equation*}
  P_T(g,s_1,\dots,s_m)\:=\:\sum_{\gamma\in\Gamma_\infty\backslash \Gamma} \rho(J(\gamma g,i))^{-1}
  \prod_{j=1}^m \trace((TY)^{[j]})^{s_j}\cdot
  e^{2\pi i\trace(T\gamma g\cdot i)}\:.
 \end{equation*}
Here $Y^{[j]}$ denotes the $j$-th alternating power of $Y$, i.e. a matrix of size $\binom{m}{j}$
 with entries the $(j\times j)$-minors of $Y$.
The convergence of these series in $(s_1,\dots,s_m)$ follows from that of the above in $(\tilde s_1,\tilde s_2)=(\sum_{j<m}j\cdot s_j,s_m)$, because $\trace(Y^{[q]})\leq\trace(Y)^q$.
We include a notion of non-holomorphic Poincar\'e series in order to give a clue how holomorphic continuation for small weights may be obtained.
However, the spectral theoretic strategy of applying adequate Casimir operators to obtain the continuations by resolvents, is involved because the higher derivatives belong to higher dimensional spaces.

For the proof of proposition~\ref{konvergenz_poincare} we use the following result.
\begin{thm}\cite[theorem~4.3]{holproj}\label{altes_konvergenz_resultat}
The series
\begin{equation*}
S_T(g,k_1,k_2)\:=\: \sum_{\gamma\in\Gamma_\infty\backslash \Gamma}\exp(2\pi i\trace(T \gamma g\cdot i))
\trace(T\im (\gamma g\cdot i))^{k_1}\det(\im(\gamma g\cdot i))^{k_2}
\end{equation*}
converges absolutely and uniformly on compact sets in the  cone 
\begin{equation*}
 \left\{(k_1,k_2)\in\CC\mid \re k_2>m \textrm{ and } \re (k_2+ \frac{k_1}{m})>m\right\}\:.
\end{equation*}
For $(k_1,k_2)$ fixed, it is absolutely bounded by a constant independent of $\tau$ and belongs to $L^2(\Gamma\backslash G)$.
\end{thm}
\begin{proof}[Proof of proposition~\ref{konvergenz_poincare}]
 For $g\in G$ let $Z=X+iY=g\cdot i\in\mathcal H$. There exists
 \begin{equation*}
  g_Z\:=\:\begin{pmatrix} Y^\frac{1}{2}&U\\0&Y^{-\frac{1}{2}}
                     \end{pmatrix}\:\in G\:,
 \end{equation*}
where $Y^\frac{1}{2}$ is the symmetric positive definite square root of $Y$,
such that $g_Z\cdot i=Z$ and such that and $g=g_Zk$ for some $k$ in the maximal compact subgroup $K$ of $G$.
Further, there exists $k_1\in\SO(m)$ such that $D=k_1Y^\frac{1}{2}k_1'$ is diagonal, $D=\mathop{diag}(d_1,\dots,d_m)$ for positive eigenvalues $d_j$  of $Y^\frac{1}{2}$.
We compute
\begin{equation*}
 \rho(J(g,i))^{-1}\:=\:\rho(J(g_Z,i)J(k,i))^{-1}\:=\:\rho(J(k,i)^{-1})\rho(Y^\frac{1}{2})\:=\:\rho(J(k,i)^{-1}k_1')\rho(D)\rho(k_1)\:.
\end{equation*}
For computing the norm $\lvert\!\lvert \rho(J(g,i))^{-1}v\rvert\!\rvert$ for a vector $v\in V_\rho$, unitary factors $\rho(k)$ for $k\in U(m)$ don't fall into account, so
\begin{equation*}
 \lvert\!\lvert \rho(J(g,i))^{-1}v\rvert\!\rvert \:\leq \:\lvert\!\lvert \rho(D)\rvert\!\rvert\cdot \lvert\!\lvert v\rvert\!\rvert\:.
\end{equation*}
We seize the operator norm $\lvert\!\lvert \rho(D)\lvert\!\lvert$.
The action of the diagonal matrix $D$ on $V_\rho$ is determined by the weights $\lambda=(\lambda_1,\dots,\lambda_m)$ of $\rho$. 
For the absolute weight $\kappa\geq 0$  of $\rho$ we have $\lambda_j-\kappa\geq 0$, $j=1,\dots,m$, and there is $j$ such that $\lambda_j=\kappa$.
If $v$ is a normalized weight vector for $\lambda$, then
\begin{equation*}
 \lvert\!\lvert \rho(D)v\rvert\!\rvert \:=\: \prod_{j=1}^m d_j^{\lambda_j}\:= \:\prod_{j=1}^m d_j^{\lambda_j-\kappa}\cdot \det(D)^\kappa\:\leq\:
 \trace(Y)^{\frac{1}{2}\sum_j(\lambda_j-\kappa)} \cdot\det(Y)^\frac{\kappa}{2}\:.
\end{equation*}
For dominant weights $\lambda$ we have $\lambda_1\geq\lambda_2\geq\dots\geq\lambda_m=\kappa\geq 0$, 
and $\lambda=l-\sum_{\alpha_i }n_i\alpha_i$ for some integers $n_i\geq 0$ and the simple roots $\alpha_i$
of $\mathfrak{gl}_m$. 
Any other weight is a conjugate of a dominant one under the Weyl group, which consists of permutations of the coordinates. So for all weights $\lambda$ of $\rho$ we have
\begin{equation*}
 0\:\leq\:\sum_{j=1}^m (\lambda_j-\kappa)\:\leq\: \sum_{j=1}^m (l_j-\kappa)\:.
\end{equation*}
Accordingly, the operator norm is seized by
\begin{equation*}
 \lvert\!\lvert \rho(D)\rvert\!\rvert \:\leq \:\trace(Y)^{\frac{1}{2}\sum_j(l_j-\kappa)} \cdot\det(Y)^\frac{\kappa}{2}\:.
\end{equation*}
So the absolute series of $S_T(g,s_1+\frac{1}{2}\sum_j(l_j-\kappa),s_2+\frac{\kappa}{2})$ in Theorem~\ref{altes_konvergenz_resultat} dominates 
$\lvert\!\lvert P_T(g,s_1,s_2)\cdot v\rvert\!\rvert$, and the claim follows from Theorem~\ref{altes_konvergenz_resultat}.
\end{proof}

\subsection{Lie algebra action}
We make sure that
 the Poincar\'e series transform adequately under the 
action of the Lie algebra $\gg_\CC=\mathfrak{sp}_{m,\CC}$. Following~\cite{holproj} we choose the following basis of 
$\mathfrak g_\CC=\mathfrak p_+\oplus\mathfrak p_-\oplus\mathfrak k_\CC$,
where $\mathfrak k_\CC$ is the Lie algebra of $K$ given by the matrices satisfying
\begin{equation*}
 \begin{pmatrix}A&S\\-S&A\end{pmatrix}\:, \quad A'=-A\:, \quad S'=S\:,
\end{equation*}
and
\begin{equation*}
 \mathfrak p_\pm=\left\{\begin{pmatrix}X&\pm iX\\\pm iX&-X\end{pmatrix},\quad X'=X\right\}.
\end{equation*}
Let $e_{kl}\in M_{m,m}(\mathbb C)$ be the elementary matrix having entries $(e_{kl})_{ij}=\delta_{ik}\delta_{jl}$ 
and let $X^{(kl)}=\frac{1}{2}(e_{kl}+e_{lk})$.
 The elements 
\begin{equation*}
 (E_\pm)_{kl}\:=\:(E_\pm)_{lk}
\end{equation*}
of $\mathfrak p_\pm$ are defined to be those corresponding to $X=X^{(kl)}$, 
$1\leq k,l\leq m$. 
Then $(E_\pm)_{kl}$, $1\leq k\leq l\leq m$ form a basis of $\mathfrak p^\pm$.
A basis of $\mathfrak k_\CC$  is given by $B_{kl}$, for $1\leq k,l\leq m$, where  $B_{kl}$ corresponds 
to 
\begin{equation*}
 A_{kl}\:=\:\frac{1}{2}(e_{kl}-e_{lk})\:\textrm{ and }\: S_{kl}\:=\:\frac{i}{2}(e_{kl}+e_{lk})\:.
\end{equation*}
For abbreviation, let $E_\pm$ be the matrix having entries $(E_\pm)_{kl}$. 
Similarly, let $B=(B_{kl})_{kl}$ be the matrix with entries $B_{kl}$ and let $B^\ast$ be its transpose 
having entries $B_{kl}^\ast=B_{lk}$.

Let us recall some facts on derivatives.
In order to compute the action of $\gg_\CC$ on  $(\rho,V_\rho)$-valued functions, we must evaluate the total differential $D\rho$ at various places $A$.
For $A$ in $\GL(m,\CC)$ let us denote by $m_A$  the multiplication in $\GL(m,\CC)$ by $A$ from the left, $m_A(g)=Ag$, respectively $m_{\rho(A)}(G)=\rho(A)G$ in $\GL(V_\rho)$.
Then we can compute the differential of $\rho\circ m_A=m_{\rho(A)}\circ\rho$ in $\mathbf 1_m=\id_{\GL(m,\CC)}$ in two different ways.
\begin{equation*}
 D(\rho\circ m_A)\mid_{\mathbf 1_m}\:=\: D\rho\mid_{m_A(\mathbf 1_m)}\circ Dm_A\mid_{\mathbf 1_m}=D\rho\mid_A\circ m_A\:, 
\end{equation*}
respectively,
\begin{equation*}
 D(m_{\rho(A)}\circ\rho)\mid_{\mathbf 1_m}\:=\: D(m_{\rho(A)})\mid_{\mathbf 1_m}\circ D\rho\mid_{\mathbf 1_m}\:=\: m_{\rho(A)}\circ d\rho\:,
\end{equation*}
where $d\rho$ is the differential of  $\rho$ at the identity, i.e. the corresponding Lie algebra representation.
It follows that
\begin{equation*}
 D\rho(A)\:=\: D\rho\mid_A\:=\:m_{\rho(A)}\circ d\rho\circ m_A^{-1}\:.
\end{equation*}
Accordingly, for a $\GL(m,\CC)$-valued $C^\infty$-function $A(t)$ we have
\begin{equation*}
 \frac{d}{dt}\rho(A(t))\mid_{t=0}\:=\: d\rho\left(A(0)^{-1}\frac{d}{dt}A(t)\mid_{t=0}\right)\circ\rho(A(0))\:.
\end{equation*}
We are specially interested in the actions $X\rho(j(g,i)^{-1})$ for Lie algebra elements $X$.
For elements $X$ of the real Lie algebra $\gg=\gg_\RR$, this action is given by
\begin{equation*}
  X\rho(J(g,i)^{-1})\:=\:\frac{d}{dt}\rho(J(g\exp(tX),i)^{-1})\mid_{t=0}\:.
\end{equation*}
For elements of the complex Lie algebra we obtain the action by  putting together the actions of the real and the imaginary part.
Recalling that the differential  of the inverse mapping $f(g)=g^{-1}$ is given by $Df(g)=-g^{-2}$, we find
\begin{equation*}
 X\rho(J(g,i)^{-1})\:=\:-d\rho\left(J(g,i)^{-1}\cdot XJ(g,i)\right)\circ\rho(J(g,i)^{-1}) \:.
\end{equation*}
We often use the abbreviation $J=J(g,i)$.
Recalling the actions of the basis elements,
\begin{eqnarray*}
 B_{ab}J(g,i)&=& J(g,i)e_{ab}\:,\\
 (E_-)_{ab}J(g,i)&=& 0\:,\\
 (E_+)_{ab}J(g,i)&=& -2J^{-1}\bar J X^{(ab)}\:,
\end{eqnarray*}
we obtain
\begin{eqnarray}
 B_{ab}\rho(J(g,i)^{-1})&=& -d\rho(e_{ab})\circ \rho(J(g,i)^{-1})\:,\label{B-drho-action}\\
 (E_-)_{ab}\rho(J(g,i)^{-1})&=& 0\:,\\
 (E_+)_{ab}\rho(J(g,i)^{-1})&=& +2d\rho(J^{-1}\bar J X^{(ab)})\circ\rho(J(g,i)^{-1})\:.
\end{eqnarray}
Here $k=J(\tilde k,i)\in U(m)$ is the image of the $K$-component $\tilde k$ of $g$ with respect to the decomposition $g=\tilde g_z\cdot \tilde k$, where
\begin{equation*}
 \tilde g_Z\:=\:\begin{pmatrix} S&U\\0&S^{-T}\end{pmatrix}
\end{equation*}
with a lower triangular matrix $S$ such that $g\cdot i=\tilde g_Z\cdot i=Z$, i.e. $SS'=Y=\im(g\cdot i)$.

Now we give the action of the Lie algebra basis on the summands
\begin{equation*}
 H_T(g,s_1,s_2)\:=\: \rho(J(g,i)^{-1}) h_T(g\cdot i,s_1,s_2)\:
\end{equation*}
of the $\rho$-valued Poincar\'e series. Here we abbreviate
\begin{equation*}
 h_T(Z,s_1,s_2)\:=\:\trace(TY))^{s_1}\det(Y))^{s_2}e^{2\pi i\trace(TZ)}\:.
\end{equation*}
Recalling the results of ~\cite[Lemma~7.1]{holproj}, we obtain
\begin{equation*}
 B_{ab}H_T(g,s_1,s_2)\:=\: - d\rho(e_{ab})\circ\rho(J(g,i)^{-1})\cdot h_T(g\cdot i,s_1,s_2)\:,
\end{equation*}
\begin{eqnarray*}
 (E_-)_{ab}H_T(g,s_1,s_2)&=& \rho(J(g,i)^{-1})\cdot 2s_1(k'S'TSk)_{ab}\cdot h_T(g\cdot i,s_1-1,s_2)\\
 &&+\rho(J(g,i)^{-1})\cdot 2s_2(k'k)_{kl}\cdot h_T(g\cdot i,s_1,s_2)\:,
\end{eqnarray*}
and 
\begin{eqnarray*}
(E_+)_{ab}H_T(g,s_1,s_2)&=&  + d\rho(J^{-1}\bar J2X^{(ab)})\circ\rho(J^{-1})\cdot    h_T(g\cdot i,s_1,s_2)    \\
&& +\rho(J^{-1})\cdot (2s_2(J^{-1}\bar J)_{ab}-8\pi(\bar JYTY\bar J)_{ab})\cdot h_T(g\cdot i,s_1,s_2)\\
&& +\rho(J^{-1})\cdot 2s_1(\bar JYTY\bar J)_{ab}\cdot h_T(g\cdot i,s_1-1,s_2)\bigr)\:.
\end{eqnarray*}
Notice that each  component of $\bar JYTY\bar J$ can be sized by $\trace(TY)$, and that  terms in $k\in U(m)$ only vary in  compact sets. Also, $d\rho(e_{ab})$ and $d\rho(X^{(ab)})$ are 
linear transformations of $V_\rho$.
So the norm of each single term of the above can be sized up to a global constant by  the norm of $H_T(g,s_1,s_2)$.
We conclude that the Poincar\'e series allow termwise differentiations:
\begin{prop}\label{prop-derivatives}
The derivatives 
 \begin{equation*}
  XP_T(g,s_1,s_2)\:=\: \sum_{\gamma\in\Gamma_\infty\backslash \Gamma}XH_T(\gamma g,s_1,s_2)
 \end{equation*}
by elements $X$ of the enveloping Lie algebra $\mathfrak U(\gg_\CC)$ have the same convergency properties  as the Poincar\'e series themselves.

In particular, in the case of large weight $\kappa>2m$, the Poincar\'e series converge in $(s_1,s_2)=(0,0)$, and vanish under the action of $E_-$.
\end{prop}

\section{Functions on the Siegel upper halfspace}\label{sec-on-H}
Let $G=\Sp(m,\RR)$ be the symplectic group of genus $m$.
We identify the maximal compact subgroup $K$ (stabilizer of $i$) with the unitary group $U(m)$ by 
\begin{equation*}
 k\:=\:\begin{pmatrix} C&S\\-S&C\end{pmatrix}\:\mapsto J(k,i)\:=\: C-iS\:.
\end{equation*}
For abbreviation, let $J(g)=J(g,i)$ for $g\in G$.
Let $C^\infty(\mathcal H,V_\rho)$ be the space of $C^\infty$-functions on $\mathcal H$ with values in the space $V_\rho$, 
and let $C^\infty(G,V_\rho)=C^\infty (G)\otimes V_\rho$.
  There is a monomorphism
 \begin{eqnarray*}
  C^\infty(\mathcal H,V_\rho) &\to& C^\infty (G, V_\rho)_\tau\:,\\
  f(Z)&\mapsto& F(G)=\rho^{-1}(J(g))F(gK i)\:.  
 \end{eqnarray*}
The images have the following transformation property under $K$ 
\begin{equation*}
 F(gk)\:=\:\rho^{-1}(J(gk)) f(gkKi)\:=\: \rho^{-1}(J(k))F(g)\:,
\end{equation*}
so they belong to $C^\infty (G, V_\rho)_\tau$, the subspace of  functions 
in $C^\infty(G,V_\rho)$ on which the action of $K$ by right translations is given by $\tau=\rho^{-1}\circ J$, and the map above implies an isomorphism
\begin{equation*}
 \phi: C^\infty(\mathcal H,V_\rho)\quad\tilde\longrightarrow\quad C^\infty (G, V_\rho)_\tau\:.
\end{equation*}
In particular, we have $F(g_Z) =\rho(Y^{1/2})f(Z)$.  
Under $\phi$ the action of the anti-holomorphic differential operator $\partial_{\bar Z}$ transforms to the action of $E_-$.
\begin{prop}\label{H-translation}
Let $\rho$ be an irreducible representation of $\GL(m,\CC)$ of highest weight $l$ and absolute weight $\kappa>2m$.
 The Poincar\'e series  
 \begin{equation*}
  p_T(Z)\:=\: \sum_{\gamma\in\Gamma_\infty\backslash \Gamma} \rho(J(\gamma,Z))^{-1}e^{2\pi i\trace(T\gamma Z)}
 \end{equation*}
 converge absolutely and locally uniformly. They are square-integrable and holomorphic. In particular, they belong to the space $[\Gamma,\rho]_0$ of holomorphic cuspforms.
 \end{prop}
\begin{proof}[Proof of proposition~\ref{H-translation}]
 Because $p_T(s_1,s_2)=\phi^{-1}(P_T(s_1,s_2))$, this is a direct consequence of proposition~\ref{konvergenz_poincare} along with proposition~\ref{prop-derivatives}.
\end{proof}

\subsection{Petterson scalar product}
For $f,h\in[\Gamma,\rho]_0$ we define the Petterson scalar product 
\begin{equation*}
 \langle f,h\rangle\::=\:\int_{\mathcal F} f(Z)'\rho(\im Z) \overline{h(Z)}~dV_{inv}\:,
\end{equation*}
where
\begin{equation*}
 dV_{inv}\:=\: \frac{dX}{\det(Y)^{\frac{m+1}{2}}}\frac{dY}{\det(Y)^{\frac{m+1}{2}}}
\end{equation*}
is the invariant measure on $\mathcal H$. Here $dX=\prod_{i\leq j}dx_{ij}$, and likewise $dY$.
We also fix the invariant measure
\begin{equation*}
 dY_{inv}\:=\:\frac{dY}{\det(Y)^{\frac{m+1}{2}}}
\end{equation*}
on the space of positive definite matrices.
Using the isomorphism $\phi$, the Petterson scalar product equals the $L^2$-scalar product on group level if one uses the normalization $dV_{inv} dk=dg$ for the Haar measures involved.
\begin{align*}
 \langle f, h\rangle \:=\:& \int_{\mathcal F}f(Z)'\rho(\im Z)\overline{h(Z)}~dV_{inv}\\
 \:=\:& \int_{\mathcal F} F(g)'\rho(J(g))'\rho(\im Z)\overline{\rho(J(g))}\overline{H(g)}~dV_{inv}\\
 \:=\:&\int_{\Gamma\backslash G}F(g)'\overline{H(g)}~dg\\
 \:=\:&\langle\!\langle F,H\rangle\!\rangle_{L^2(\Gamma\backslash G)}\:.
\end{align*}
Here we used $Z=g\cdot i$ and the formula $\im MZ=(CZ+D)'^{-1}\im(Z)\overline{(CZ+D)^{-1}}$.
\subsection{Unfolding the Poincar\'e series}
Let $f$ be a (non-homomorphic) modular form of weight $\rho$. We have
\begin{align*}
 \langle f, P_Tv\rangle \:=\:&
 \int_\mathcal F f(Z)'\rho(\im Z)\sum_{\gamma\in\Gamma_\infty\backslash \Gamma} \overline{\rho^{-1}(J(\gamma,z))}e^{-2\pi i\trace(T\gamma \bar Z)}v~dV_{inv}\\
 \:=\:& 
 \int_\mathcal F \sum_{\gamma\in\Gamma_\infty\backslash \Gamma}\!\!f(\gamma Z)'\rho^{-1}(J(\gamma,Z))'\rho(\im Z)\overline{\rho^{-1}(J(\gamma,Z))} e^{-2\pi i\trace(T\gamma \bar Z)}v~dV_{inv}\\
 \:=\:&
 \int_{\Gamma\backslash \mathcal H}\sum_{\gamma\in\Gamma_\infty\backslash \Gamma}f(\gamma Z)'\rho(\im(\gamma Z))e^{-2\pi i\trace(T\gamma \bar Z)}v~dV_{inv}\\
\:=\:&
 \int_{\Gamma_\infty\backslash \mathcal H}f(Z)'\rho(\im Z)e^{-2\pi i\trace(T\bar Z)}v~dV_{inv}\:.
\end{align*}
More correctly, we must restrict to the case of forms of moderate growth, which means that the above integral exists.
Assuming $f$ to have Fourier expansion
\begin{equation*}
 f(Z)\:=\:\sum_{\tilde T}\rho(\tilde T^{\frac{1}{2}})a(\tilde T,\tilde T^{\frac{1}{2}}Y\tilde T^{\frac{1}{2}})e^{2\pi i\trace(\tilde TX)}\:,
\end{equation*}
(notice that the vector valued coefficients are well-defined because $\rho(\tilde T^{\frac{1}{2}})$ belongs to $\GL(V_\rho)$ and 
$a(\tilde T,\tilde T^{\frac{1}{2}}Y\tilde T^{\frac{1}{2}})$ belongs to $V_\rho$)
we calculate further
\begin{align*}
 \langle f, P_Tv\rangle \:=\:&
 \int_{Y>0}a(T,T^{\frac{1}{2}}YT^{\frac{1}{2}})'\rho(T^{\frac{1}{2}})\rho(Y)v e^{-2\pi\trace(TY)}\frac{dY_{inv}}{\det(Y)^{\frac{m+1}{2}}}\\
 \:=\:&
 \det(T)^{\frac{m+1}{2}}\int_{Y>0}a(T,Y)'\rho(Y)\rho(T^{-\frac{1}{2}})v e^{-2\pi\trace(Y)}\frac{dY_{inv}}{\det(Y)^{\frac{m+1}{2}}}\:.
\end{align*}
If $f$ is assumed to be holomorphic,  we may write for its Fourier expansion
\begin{equation*}
 f(Z)\:=\:\sum_{\tilde T}\rho(\tilde T^{1/2})a(\tilde T)e^{2\pi i\trace(\tilde TZ)}\:,
\end{equation*}
where
\begin{equation*}
 a(\tilde T)\:=\: a(\tilde T,\tilde T^\frac{1}{2}Y\tilde T^\frac{1}{2})\cdot e^{2\pi\trace(\tilde TY)}
\end{equation*}
is independent of $Y$.
Then we obtain
\begin{align*}
 \langle f, P_Tv\rangle \:=\:&
 \det(T)^{\frac{m+1}{2}}a(T)'\int_{Y>0}\rho(Y)\rho(T^{-\frac{1}{2}})v e^{-4\pi\trace(Y)}\frac{dY_{inv}}{\det(Y)^{\frac{m+1}{2}}} \:.
\end{align*}
\subsection{Sturm's operator}\label{sturm_definition}
For  Sturm's operator to reproduce holomorphic cuspforms we must normalize it such that this last expression is $a(T)'\cdot v$.
So we are in due to calculate the integrals
\begin{equation*}
\Gamma(\rho)\:=\:\int_{Y>0}\rho(Y) e^{-4\pi\trace(Y)}~dY_{inv}\:
\end{equation*}
for  varying $\rho$. 
For $\rho$ of large enough absolute weight, this Gamma integral is convergent and belongs to $\mathop{End}(V_\rho)$ (\cite{godement}). It allows analytic continuation to smaller weights.
We expect $\Gamma(\rho)$ to be invertible in general apart from a discrete set of zeros and poles and  prove this for a class of representations in section~\ref{section_gamma_integral}.

For all $\rho$ such that the following is well-defined as an element of $\GL(V_\rho)$ let
\begin{equation*}
 C(\rho)\:=\:\int_{Y>0}\rho(Y) e^{-4\pi\trace(Y)}\frac{dY_{inv}}{\det(Y)^{\frac{m+1}{2}}}\:=\:
  (4\pi)^{(m+1)-\sum_j l_j}\cdot
  \Gamma(\rho\otimes\det\nolimits^{-\frac{m+1}{2}})\:.
\end{equation*}
Then define the normalized Sturm operator by
\begin{equation*}
 \mathop{St}\nolimits_\rho(f)\:=\:\sum_{T>0}\rho(T^\frac{1}{2})b(T)e^{2\pi i\trace(TZ)}\:,
\end{equation*}
where $b(T)$ is defined by
\begin{equation*}
b(T)'\:=\: \det(T)^{-\frac{m+1}{2}}\int_{Y>0}a(T,Y)'\rho(T^\frac{1}{2})C(\rho)^{-1}\rho(Y)\rho(T^{-\frac{1}{2}}) \frac{e^{-2\pi\trace(Y)}dY_{inv}}{\det(Y)^{\frac{m+1}{2}}}\:.
\end{equation*}
Then, for holomorphic input $f$ as above and $v\in V_\rho$ we obtain $b(T)=a(T)$.
The unfolding process above proves theorem~\ref{theorem_hol_proj}. 
The assumption that  $\Gamma(\rho\otimes\det\nolimits^{-\frac{m+1}{2}})$ is an automorphism  is satisfied for example for alternating powers 
$\rho=\mathop{st}^{[q]}\det\nolimits^\kappa$, $\kappa>m-1$ (see proposition~\ref{gamma_alternating}).


\section{Phantom terms by Sturm's operator}\label{sec-sturm}
We will prove theorem~\ref{thm-phantom}. So  fix rank $m=2$.
We test Sturm's operator in case of $\rho$ being the representation of minimal $K$-type $(\kappa+1,\kappa)$.
We show that in analogy to the case of scalar weight $\kappa$ the Maass shift of cusp forms $h\in [\Gamma,(\kappa-1,\kappa-2)]_0$ produce phantom terms if 
and only if $\kappa=3$.

we have
\begin{equation*}
 C(\rho)\:=\:(4\pi)^{3-(2\kappa+1)}(\kappa-\frac{3}{2})\Gamma_2(\kappa-\frac{3}{2})\mathbf 1_2\:.
\end{equation*}
Let $c(\rho)$ be the scalar  such that $C(\rho)=c(\rho)\mathbf 1_2$.
Let $k=\kappa-2$ and let $h\in[\Gamma,\tau]_0$ be a holomorphic cuspform for $\tau=(k+1,k)$ with Fourier expansion
\begin{equation*}
 h(Z)\:=\:\sum_{T>0} \tau(T^\frac{1}{2})a(T)e^{2\pi i\trace(TZ)}\:.
\end{equation*}
Maass' shift operator is given by (see~\cite[5.1]{phantom})
\begin{equation*}
 \Delta_+^{[2]} h(Z)\:=\:(2i)^2(\tau\otimes\det\nolimits^{-\frac{1}{2}})(Y^{-1})\cdot\det(\partial_Z)\left((\tau\otimes\det\nolimits^{-\frac{1}{2}})(Y)h(Z)\right)\:.
\end{equation*}
The image of $h$ under $\Delta_+^{[2]}$ is a non-holomorphic form of weight $\tau\otimes\det^2$, i.e. $(k+3,k+2)=(\kappa+1,\kappa)$. Hence (see~\cite{phantom}), its holomorphic projection is zero $\mathop{pr}_{hol}(\Delta_+^{[2]}(h))=0$.
We show that Sturm's operator $\mathop{St}_{\tau\otimes\det^2}(\Delta_+^{[2]}(h))$ is non-zero if and only if $k=1$.
 For to apply Maass' operator to $h$ it is enough to apply it to $e^{2\pi i\trace(TZ)}$.
 Here $(\tau\otimes\det\nolimits^{-\frac{1}{2}})(Y)=\det(Y)^{k-\frac{1}{2}} Y$. Let $f(Z)=\det(Y)^{k-\frac{1}{2}}e^{2\pi i\trace(TZ)}$ and $g(Z)=Y$.
By  \cite[p.~211]{freitag}  we have
\begin{equation*}
 \det(\partial_Z)(f\cdot g)\:=\:\det(\partial_Z)(f)\cdot g +2(\partial_Z(f)\sqcap \partial_Z(g))+f\cdot\det(\partial_Z)(g)\:.
\end{equation*}
Here the last term is zero, because $\det(\partial_Z)$ is a differential operator of homogeneous degree two and $g(Z)=Y$ is of degree one.
For the first term we obtain following~\cite[5.2]{phantom}
\begin{eqnarray*}
 \det(\partial_Z)(f(Z))\cdot Y&=& -\frac{1}{4}k(k-\frac{1}{2})\det(Y)^{k-\frac{3}{2}}e^{2\pi i\trace(TZ)}\cdot Y\\
 &&-\frac{i}{2}(k-\frac{1}{2})(2\pi i)\trace(TY)\det(Y)^{k-\frac{3}{2}}e^{2\pi i\trace(TZ)}\cdot Y\\
 &&+(2\pi i)^2\det(Y)^{k-\frac{1}{2}}\det(T)e^{2\pi i\trace(TZ)}\cdot Y\:.
\end{eqnarray*}
For the second term we find
\begin{equation*}
 \partial_Z(f(Z))\:=\:-\frac{i}{2}(k-\frac{1}{2})\det(Y)^{k-\frac{1}{2}}e^{2\pi i\trace(TZ)}\cdot Y^{-1} +2\pi i\det(Y)^{k-\frac{1}{2}}e^{2\pi i\trace(TZ)}\cdot T\:,
\end{equation*}
and $\partial_Z (Y_{jk})=-\frac{i}{2}X^{(jk)}$.
Here $X^{(jk)}=\frac{1}{2}(e_{jk}+e_{kj})$.
So the second  term $2(\partial_Z(f(Z))\sqcap \partial_Z(g(Z))$  equals
\begin{eqnarray*}
&&
 \sum_{j,k} \left(-\frac{1}{4}(k-\frac{1}{2})\det(Y)^{k-\frac{3}{2}} e^{2\pi i\trace(TZ)}2(Y^{-1}\sqcap X^{(jk)} )\cdot X^{(jk)}\right.\\
 &&\hspace*{2cm}
 \left.
 +\quad \pi\det(Y)^{k-\frac{1}{2}}e^{2\pi i\trace(TZ)}\cdot 2(T\sqcap X^{(jk)})\cdot X^{(jk)}\right)\:,\\
\end{eqnarray*}
which by definition of $\sqcap$-multiplication (\cite[p.~207]{freitag}) is
\begin{equation*}
  -\frac{1}{4}(k-\frac{1}{2})\det(Y)^{k-\frac{3}{2}}e^{2\pi i\trace(TZ)}\cdot Y +\pi\det(Y)^{k-\frac{1}{2}}e^{2\pi i\trace(TZ)}\det(T)\cdot T^{-1}\:.
\end{equation*}
Altogether we obtain
\begin{align*}
 \Delta_+^{[2]}(e^{2\pi i\trace(TZ)})
 \:=\:& (k+1)(k-\frac{1}{2})\frac{1}{\det(Y)}e^{2\pi i\trace(TZ)}\cdot \mathbf 1_2 \\
 & 
 -4\pi (k-\frac{1}{2})\frac{\trace(TY)}{\det(Y)}e^{2\pi i\trace(TZ)}\cdot \mathbf 1_2\\
 &
 +(4\pi)^2\det(T)e^{2\pi i\trace(TZ)}\cdot \mathbf 1_2\\
 &
 -4\pi\det(T)e^{2\pi i\trace(TZ)}\cdot(TY)^{-1}\:.
\end{align*}
The Fourier coefficients of 
\begin{equation*}
 \tilde h(Z)\:=\:\Delta_+^{[2]}(h)\:=\: \sum_{T>0} \rho(T^\frac{1}{2})a(T,T^\frac{1}{2}YT^\frac{1}{2})e^{2\pi i\trace(TX)}\:
\end{equation*}
are given by $\rho(T^\frac{1}{2})a(T,T^\frac{1}{2}YT^\frac{1}{2})$, which equal
\begin{align*}
 &e^{-2\pi\trace(TY)}\cdot\Bigl(\frac{(k-\frac{1}{2})(k+1)}{\det(Y)}-4\pi (k-\frac{1}{2})\frac{\trace(TY)}{\det(Y)}
 +(4\pi)^2\det(T)\Bigr)\cdot 
\tau(T^\frac{1}{2})a(T)\\
 &- e^{-2\pi\trace(TY)}\cdot 4\pi\det(T)\cdot Y^{-1}T^{-1}\tau(T^\frac{1}{2})a(T)\:,
\end{align*}
respectively $a(T,Y)$ given by
\begin{eqnarray*}
  &&e^{-2\pi\trace(Y)}\cdot\left(\frac{(k-\frac{1}{2})(k+1)}{\det(Y)}-4\pi (k-\frac{1}{2})\frac{\trace(Y)}{\det(Y)}+(4\pi)^2\right)\cdot 
  \rho(T^{-\frac{1}{2}})a(T)\\
 &&-  e^{-2\pi\trace(Y)}\cdot4\pi\cdot Y^{-1}a(T)\:.
 \end{eqnarray*}
 Accordingly, for to compute Sturm's operator we evaluate the sum of the following terms up to the factor $\det(T)^{-\frac{1}{2}}c(\rho)^{-1}\cdot a(T)'$. First,
 \begin{equation*}
  (k-\frac{1}{2})(k+1) \int_{Y>0}Y\det(Y)^{k+2+s-\frac{5}{2}}e^{-4\pi\trace(Y)}~dY_{inv}\cdot 
 \end{equation*}
 which by Proposition~\ref{gamma_standard} equals
 \begin{equation}\label{sturm_term_1}
  (k-\frac{1}{2})(k+1)(4\pi)^{-2(s+k)}(s+k-\frac{1}{2})\Gamma_2(s+k-\frac{1}{2})\mathbf 1_2\:.
 \end{equation}
 Second,
 \begin{equation*}
  -4\pi(k-\frac{1}{2}) \int_{Y>0}Y\trace(Y)\det(Y)^{k+2+s-\frac{5}{2}}e^{-4\pi\trace(Y)}~dY_{inv}
 \end{equation*}
which by Lemma~\ref{polynoimal-integrals} equals
\begin{equation}\label{sturm_term_2}
 -2(k-\frac{1}{2})(4\pi)^{-2(s+k)}(s+k-\frac{1}{2})(s+k)\Gamma_2(s+k-\frac{1}{2})\mathbf 1_2\:.
\end{equation}
Third,
\begin{equation*}
 (4\pi)^2 \int_{Y>0}Y\det(Y)^{k+2+s-\frac{3}{2}}e^{-4\pi\trace(Y)}~dY_{inv}
\end{equation*}
which by Proposition~\ref{gamma_standard} equals
\begin{equation}\label{sturm_term_3}
 (4\pi)^{-2(s+k)}(s+k+\frac{1}{2})\Gamma_2(s+k+\frac{1}{2})\mathbf 1_2\:.
\end{equation}
And
\begin{equation}\label{sturm_term_4}
4\pi \cdot\!\int_{Y>0}\!\!\mathbf 1_2\det(Y)^{k+2+s-\frac{3}{2}}e^{-4\pi\trace(Y)}~dY_{inv}\:=\:(4\pi)^{-2(s+k)}\Gamma_2(s+k+\frac{1}{2})\mathbf 1_2\:.
\end{equation}
According to (\ref{sturm_term_1})--(\ref{sturm_term_4}),  Sturm's operator  applied to $\Delta_+^{[2]}h(Z)$ is given in terms of coefficients by the limit $\lim_{s\to 0} b(T,s)$,
where 
\begin{equation*}
 b(T,s)\:=\:\det(T)^{-\frac{3}{2}}c(\rho)^{-1}(4\pi)^{-2(s+k)}\frac{s^2-\frac{s}{2}}{(s+k-1)}\Gamma_2(s+k+\frac{1}{2})a(T)\:.
\end{equation*}
Here we used the identity $(s+k-\frac{1}{2})\Gamma_2(s+k-\frac{1}{2})=(s+k-1)^{-1}\Gamma_2(s+k+\frac{1}{2})$.
The limit 
\begin{equation*}
 \lim_{s\to 0} b(T,s)\:=\: (4\pi)^2\det(T)^{-\frac{3}{2}} \cdot \lim_{s\to 0}\frac{s^2-\frac{s}{2}}{s+k-1}\cdot a(T)
 \end{equation*}
 is zero in all cases $k>1$, and equals 
\begin{equation*}
 b(T)\:=\: -\frac{(4\pi)^2}{2}\det(T)^{-\frac{3}{2}} \cdot a(T)
\end{equation*}
in case $k=1$. So Sturm's operator  applied to $\Delta_+^{[2]}h(Z)$ is non-zero exactly in case   $\rho=(\kappa+1,\kappa)$ with $\kappa=3$, 
which is the minimal $K$-type of the holomorphic discrete series representation of Harish-Chandra parameter $(4,3)-(1,2)=(3,1)$.



\section{Gamma integrals}\label{section_gamma_integral}
For an irreducible finite dimensional representation $\rho$ of $\GL(m,\CC)$ of absolute weight $\kappa$ we are interested in the $\mathop{End}(V_\rho)$-valued integral
\begin{equation*}
\Gamma(\rho)\:=\: \int_{Y>0}\rho(Y)e^{-\trace(Y)}dY_{inv}\:.
\end{equation*}
 Introducing a factor $\det(Y)^s$ the integral
\begin{equation*}
 \Gamma(\rho\otimes\det\nolimits^s)\:=\: \int_{Y>0}\rho(Y)\det(Y)^se^{-\trace(Y)}dY_{inv}\:
\end{equation*}
exists for $\re s+\kappa>\frac{m-1}{2}$ (\cite{godement}). 
We denote by $\Gamma(\rho)$ its  meromorphic continuation to $s=0$.
Let
\begin{equation*}
 \Gamma_m(s)\:=\:\pi^{\frac{m(m-1)}{4}}\prod_{\nu=0}^{m-1}\Gamma(s-\frac{\nu}{2})
\end{equation*}
denote the classical Gamma function of level $m$ which for $\re s>\frac{m-1}{2}$ is given by the integral 
\begin{equation*}
 \Gamma_m(s)\:=\:\int_{Y>0}\det(Y)^se^{-\trace(Y)}dY_{inv}\:.
\end{equation*}
In particular we have $\Gamma(\det\nolimits^s)=\Gamma_m(s)$.
An important property of the operator integrals is their $\SO(m)$-equivariance.
\begin{lem}\label{orthogonale_invarianz}
 The  integral $\Gamma(\rho)$ is invariant under orthogonal transformations
 \begin{equation*}
 \Gamma(\rho)\:=\: \rho(k') \Gamma(\rho) \rho(k)\:,
 \end{equation*}
for all $k\in \SO(m,\RR)\subset U(m)$.
\end{lem}
\begin{proof}[Proof of Lemma~\ref{orthogonale_invarianz}]
 For $k\in \SO(m)$ we have
 \begin{equation*}
 \Gamma(\rho\otimes\det\nolimits^s)\:=\: \int_{Y>0}\rho(k'Yk)\det\nolimits^s(k'Yk) e^{-\trace(k'Yk)}~dY_{\mathop{inv}}\:=\: \rho(k')\Gamma_m(\rho\otimes\det\nolimits^s)\rho(k)\:.
 \end{equation*}
 By the uniqueness of meromorphic continuation, this also holds for $\Gamma(\rho)$.
\end{proof}
\subsection{Alternating powers}\label{sec-alt}
\begin{prop}\label{gamma_alternating}\label{gamma_standard}
For $q=1,\dots,m$,
let $\mathop{st}^{[q]}$ be the $q$-th alternating power of the  standard representation of $\GL(m,\CC)$, i.e. the irreducible representation of highest weight  $(1,\dots,1,0,\dots,0)$, 
where the number of ones is $q$.
Define the polynomial $C_{[q]}(x)=x(x+\frac{1}{2})\cdots(x+\frac{q-1}{2})$.
The automorphism-valued function
 \begin{equation*}
  \Gamma(\mathop{st}\nolimits^{[q]}\otimes\det\nolimits^s)\:=\:\int_{Y>0}Y^{[q]}\det(Y)^se^{-\trace(Y)}dY_{inv}\:=\: (-1)^qC_{[q]}(-s)\Gamma_m(s)\cdot\id_{\mathop{st}^{[q]}}
 \end{equation*}
 is holomorphic on $\re s>\frac{m-1}{2}$ and has  meromorphic continuation to the complex plane,  the pole behavior being that of the scalar function $C_{[q]}(-s)\Gamma_m(s)$.
\end{prop}
\begin{proof}[Proof of Proposition~\ref{gamma_alternating}]
 For a symmetric positive definite matrix $T$ it holds
 \begin{equation}\label{gamma_grundformel}
  \int_{Y>0}\det(Y)^se^{-\trace(TY)}dY_{inv}\:=\:\det(T)^{-s}\Gamma_m(s)\:.
 \end{equation}
 Differentiating  both sides by $\partial_T^{[q]}$ we obtain (\cite[p. 210, p. 213]{freitag})
 \begin{equation*}
  (-1)^q\int_{Y>0}Y^{[q]}\det(Y)^se^{-\trace(TY)}dY_{inv}\:=\:C_{[q]}(-s)T^{-[q]}\det(T)^{-s}\Gamma_m(s)\:.
 \end{equation*}
Evaluating  at $T=\mathbf 1_m$ yields
\begin{equation*}\label{integral_mit_Y_eintraegen}
  \int_{Y>0}Y^{[q]}\det(Y)^se^{-\trace(Y)}dY_{inv}\:=\:(-1)^qC_{[q]}(-s)\Gamma_m(s)\cdot\id_{\mathop{st}^{[q]}}\:.\qedhere
 \end{equation*}
\end{proof}
By substitution $s'=s-\kappa$, Proposition~\ref{gamma_alternating} determines $\Gamma(\rho)$ for the representations $\rho=\mathop{st}^{[q]}\otimes\det\nolimits^\kappa$. 
The computation for general $\rho$ may be obtained by chasing Young tableaux, but for  rank $m>2$ we don't obtain an instructive  general formula.
This combinatorial aspect becomes visible in the formulas for $m=2$ by the involved triangle numbers $a_{n,m}$ defined in Proposition~\ref{triangle_numbers}.
\subsection{Rank two}\label{sec-Gamma-rank-2}
For  a general formula for $\Gamma(\rho\otimes\det^{-s})$ for the irreducible representations $\rho$ of $\GL(2,\CC)$ of highest weights $(r,0)$, we need some preparations.
\begin{prop}\label{triangle_numbers}
Define the following triangle numbers $a_{n,m}$ for $n,m\in\NN_0$. Let $a_{n,0}=1$ for all $n\in\NN_0$, and let $a_{0,m}=0$ for all $m>0$.
For $n,m>0$ define by recursion
 \begin{equation*}
  a_{n,m}\:=\:\bigl(n-2(m-1)\bigr)\cdot a_{n-1,m-1}+a_{n-1,m}\:.
 \end{equation*}
 The triangle numbers have the following properties.
 \begin{enumerate}
  \item [(i)]
  $a_{n,1}\:=\:\frac{1}{2}n(n+1)$.
  \item [(ii)]
  $a_{n,2}\:=\:\frac{1}{8}n(n+1)(n-1)(n-2)$.
 \item [(iii)]
  $a_{n,m}=0$ for all $m>\lfloor\frac{n+1}{2}\rfloor$ (Gauss brackets).
 \item [(iv)]
$a_{2\nu-1,\nu}\:=\:a_{2(\nu-1),\nu-1}$.
 \end{enumerate}
\end{prop}
We will be specially interested in the numbers $a_{2\nu-1,\nu}$, for which we give an explicit formula in Proposition~\ref{a_2n-1,n-coefficient}.
\begin{proof}[Proof of proposition~\ref{triangle_numbers}.]
Obviously, $a_{1,1}=a_{0,0}+a_{0,1}=1$. Assuming  $a_{n-1,1}=\frac{1}{2}n(n-1)$, we obtain property (i) for $n$ by induction and the  recursion formula
\begin{equation*}
 a_{n,1}\:=\:n\cdot a_{n-1,0}+a_{n-1,1}\:=\:\frac{1}{2}n(n+1)\:.
\end{equation*}
Property (iii) holds  for $n=0$ by definition, and by induction the right hand side of the recursion formula is zero for all $m>\lfloor\frac{n}{2}\rfloor+1$. So the single case left to check is
that of even $n=2k$ and $m=k+1$. But here the recursion yields $a_{2k,k+1}=(2k-2k)a_{2k-1,k}+a_{2k-1,k+1}=0$.
Property (ii) is also obtained by induction using (i) and (iii).
For property (iv) notice that by (iii) $a_{n-1,\frac{n+1}{2}}=0$ for odd $n$, so the recursion formula yields $a_{n,\frac{n+1}{2}}=a_{n-1,\frac{n-1}{2}}$.
\end{proof}
\begin{lem}\label{det-ableitungen}
 Let $T$ be a  symmetric two-by-two matrix variable and denote by $\partial_{ij}=\frac{1+\delta_{ij}}{2}\partial_{T_{ij}}$ the normalized partial derivatives.
 For all $n>0$ the derivatives of the function $\det(T)^{-s}$ are given by
 \begin{equation*}
  \partial_{jj}^{(n)}(\det(T)^{-s})\:=\:(-1)^nT_{ii}^n\det(T)^{-(s+n)}\prod_{l=0}^{n-1}(s+l)\:,
  \end{equation*}
 for $\{i,j\}=\{1,2\}$,
 and 
\begin{equation*}
  \partial_{12}^{(n)}(\det(T)^{-s})\:=\:\sum_{k=0}^{n-1}2^{-k}a_{n-1,k}\det(T)^{-(s+n-k)}T_{12}^{n-2k}\cdot \prod_{l=0}^{n-k-1}(s+l)\:,
 \end{equation*}
 where the  numbers $a_{n,m}$ are defined in Proposition~\ref{triangle_numbers}.
Further,
 \begin{align*}
  \partial_{11}^{(n_1)}\partial_{22}^{(n_2)}(\det(T)^{-s})=&\sum_{k=0}^{\min\{n_1,n_2\}}k!\binom{n_1}{k}\binom{n_2}{k}(-1)^{n_1+n_2+k}T_{11}^{n_2-k}T_{22}^{n_1-k}\times\\
  &\hspace*{10mm}\times
  \det(T)^{-(s+n_1+n_2-k)}\cdot\prod_{l=0}^{n_1+n_2-k-1}(s+l)\:.
 \end{align*}
 \end{lem}

\begin{proof}[Proof of lemma~\ref{det-ableitungen}.]
 Iterating $\partial_{jj}(\det(T)^{-s})=-sT_{ii}\det(T)^{-(s+1)}$ we obtain
 \begin{equation*}
  \partial_{jj}^{(n)}(\det(T)^{-s})\:=\:(-1)^nT_{ii}^n\det(T)^{-(s+n)}\prod_{l=0}^{n-1}(s+l)\:.
  \end{equation*}
  Then for $\partial_{11}^{(n_1)}\partial_{22}^{(n_2)}(\det(T)^{-s})$ we obtain
  \begin{align*}
   &
    \partial_{11}^{(n_1)}\left((-1)^{n_2}T_{11}^{n_2}\det(T)^{-(s+n_2)}\prod_{l=0}^{n_2-1}(s+l)\right)\\
    &=(-1)^{n_2}\prod_{l=0}^{n_2-1}(s+l)\sum_{k=0}^{n_1}\binom{n_1}{k}\partial_{11}^{(k)}(T_{11}^{n_2})\cdot \partial_{11}^{(n_1-k)}(\det(T)^{-(s+n_2)})\\
    &=   
   \sum_{k=0}^{\min\{n_1,n_2\}}\binom{n_1}{k}\frac{n_2!}{(n_2-k)!}(-1)^{n_1+n_2+k}T_{11}^{n_2-k}T_{22}^{n_1-k}\times\\
   &\hspace*{40mm}\times
  \det(T)^{-(s+n_1+n_2-k)}\prod_{l=0}^{n_1+n_2-k-1}(s+l)\:.
  \end{align*}
 Further, $\partial_{12}(\det(T)^{-s})=sT_{12}\det(T)^{-(s+1)}$
 as well as
 \begin{align*}
  \partial_{12}^{(2)}(\det(T)^{-s})&
  \:=\:s(s+1)T_{12}^2\det(T)^{-(s+2)}+\frac{1}{2}s\det(T)^{-(s+1)}\\
   \end{align*}
   satisfy the claimed formula. Then  $\partial_{12}^{(n+1)}$ is given by induction
\begin{align*}
& \partial_{12}\left( \sum_{k=0}^{n-1}2^{-k}a_{n-1,k}\det(T)^{-(s+n-k)}T_{12}^{n-2k}\prod_{l=0}^{n-k-1}(s+l)\right)\\
=& \sum_{k=0}^{n-1}2^{-k}a_{n-1,k}\cdot(s+n-k)\cdot\det(T)^{-(s+n+1-k)}T_{12}^{n+1-2k}\prod_{l=0}^{n-k-1}(s+l)\\
&+\sum_{k=0}^{n-1}2^{-k}a_{n-1,k}\cdot\frac{1}{2}(n-2k)\cdot\det(T)^{-(s+n-k)}T_{12}^{n+1-2(k+1)}\prod_{l=0}^{n-k-1}(s+l)\\
=&\sum_{k=0}^{n}2^{-k}a_{n,k}\det(T)^{-(s+n+1-k)}T_{12}^{n+1-2k}\prod_{l=0}^{n+1-k-1}(s+l)\:,
\end{align*}  
where we have used the product rule and the recursion formula defining the numbers~$a_{n,k}$ (see proposition~\ref{triangle_numbers}) as well as the fact  $a_{n,n}=0$ for $n\geq 2$.
\end{proof}
\begin{lem}\label{polynoimal-integrals}
 Let $n_1,n_2,n_3\geq 0$ be integers. The integral
 \begin{equation*}
  \int_{Y>0}Y_{11}^{n_1}Y_{22}^{n_2}Y_{12}^{n_3}\det(Y)^{s}e^{-\trace(Y)}dY_{inv}
 \end{equation*}
is a holomorphic function on $\re s>\frac{1}{2}$. For odd $n_3$ it is zero, while for even $n_3$ it is given by
\begin{equation*}
 \Gamma_2(s)\cdot 2^{-\frac{n_3}{2}}a_{n_3-1,\frac{n_3}{2}}\cdot\sum_{k=0}^{\min\{n_1,n_2\}}(-1)^{k}\binom{n_1}{k}\binom{n_2}{k}k!\cdot\hspace*{-5mm}
 \prod_{l=0}^{n_1+n_2+\frac{n_3}{2}-k-1}\hspace*{-5mm}(s+l)\:.
\end{equation*}
Here we put $a_{-1,0}=1$.  In particular, the integral has meromorphic continuation 
to the complex plane, the poles being at most simple and included in those of $\Gamma_2(s)$.
\end{lem}
\begin{proof}[Proof of lemma~\ref{polynoimal-integrals}]
Starting with the identity
 \begin{equation*}
  \int_{Y>0} \det(Y)^se^{-\trace(TY)}dY_{inv}\:=\:\det(T)^{-s}\Gamma_2(s) 
 \end{equation*}
 for $\re s>\frac{1}{2}$, which holds for all positive definite $T$, we differentiate both sides by $\partial_{11}^{(n_1)}\partial_{22}^{(n_2)}\partial_{12}^{(n_3)}$ to determine
 $\int_{Y>0}Y_{11}^{n_1}Y_{22}^{n_2}Y_{12}^{n_3}\det(Y)^{s}e^{-\trace(TY)}dY_{inv}$ by
 \begin{equation*}
\Gamma_2(s)\cdot (-1)^{n_1+n_2+n_3}\partial_{11}^{(n_1)}\partial_{22}^{(n_2)}\partial_{12}^{(n_3)}(\det(T)^{-s})\:.
 \end{equation*}
 Evaluating at $T=\mathbf 1_2$, we obtain a formula for the integral in question
by
\begin{equation*}
\Gamma_2(s)\cdot (-1)^{n_1+n_2+n_3}\partial_{11}^{(n_1)}\partial_{22}^{(n_2)}\partial_{12}^{(n_3)}(\det(T)^{-s})\mid_{T=\mathbf 1_2}\:.
 \end{equation*}
 Lemma~\ref{det-ableitungen} determines the derivative 
 \begin{align*}
  \partial_{11}^{(n_1)}\partial_{22}^{(n_2)}\partial_{12}^{(n_3)}(\det(T)^{-s})=&
  \sum_{k=0}^{\min\{n_1,n_2\}}\sum_{k_3=0}^{n_3-1}(-1)^{n_1+n_2-k}2^{-k_3}\binom{n_1}{k}\binom{n_2}{k}k!\times\\
  &\times T_{11}^{n_1-k}T_{22}^{n_2-k}T_{12}^{n_3-2k_3}
  \cdot \det(T)^{-(s+n_1+n_2+n_3-k-k_3)}\times\\
  &\times\prod_{l=0}^{n_1+n_2+n_3-k-k_3-1}(s+l)\:.
 \end{align*}
Evaluating  at $T=\mathbf 1_2$, the factor $T_{12}^{n_3-2k_3}$ is zero apart from the case $n_3=2k_3$.
In this case the formula reduces to the  claimed one, whereas it is zero for odd~$n_3$.
\end{proof}
Consider the explicit realization of the representation $\rho=\rho_r$ of $\GL_2(\CC)$ of highest weight $(r,0)$  on the space $\mathcal P_r$ of homogeneous polynomials of degree $r$ 
in the variable $z=(z_1,z_2)$, 
\begin{equation*}
 \rho_{r}(g)\bigl( P(z)\bigr)\:=\:P(z\cdot g)
\end{equation*} 
for $P\in \mathcal P_r$. 
We  determine  $\Gamma_2(\rho_r\otimes\det^s)$ by its action on the $K=\SO(2)$-weight spaces.
 For $k=0,1,\dots,r$ the polynomial
\begin{equation*}
 V_{k}(z)\:=\: (z_1-iz_2)^{r-k}(z_1+iz_2)^k
\end{equation*}
is a $K$-eigenfunction of weight $-r+2k$.
We find
\begin{equation*}
 V_k(z)\:=\:\sum_{\nu=0}^r z_1^{r-\nu}z_2^\nu i^\nu\sum_{j=0}^{\min\{r-k,\nu\}}(-1)^j\binom{r-k}{j}\binom{k}{\nu-j}\:,
\end{equation*}
whereas
\begin{align*}
 \rho_r(Y)V_k(z)&=\bigl[(Y_{11}-iY_{12})z_1+(Y_{22}+iY_{12})(-iz_2)\bigr]^{r-k}\bigl[(Y_{11}+iY_{12})z_1+(Y_{22}-iY_{12})iz_2\bigr]^{k}\\
 &=\sum_{\nu=0}^rz_1^{r-\nu}z_2^\nu i^\nu\cdot P_k(\nu,Y)\:,
\end{align*}
with 
\begin{align*}
 P_k(\nu,Y)=&\sum_{j=0}^{\min\{r-k,\nu\}}(-1)^j\binom{r-k}{j}\binom{k}{\nu-j}(Y_{11}-iY_{12})^{r-k-j}\\
 &\hspace*{2.5cm}\times(Y_{11}+iY_{12})^{k+j-\nu}(Y_{22}+iY_{12})^{j}(Y_{22}-iY_{12})^{\nu-j}\:.
\end{align*}
By lemma~\ref{orthogonale_invarianz}, $\Gamma_2(\rho_r\otimes\det^s)$ commutes with $K$, so 
acts by scalars on the $1$-dimensional $K$-eigenspaces.
Defining 
\begin{equation*}
 c_k(\nu)\:=\:\sum_{j=0}^{\min\{r-k,\nu\}}\binom{r-k}{j}\binom{k}{\nu-j}(-1)^j\:
\end{equation*}
the integral
\begin{equation}\label{gamma-r-k-expression}
 \Gamma(r,k,s)\:=\:\frac{1}{c_k(\nu)}\int_{\mathcal Y}P_k(\nu,Y)\det(Y)^se^{-\trace(Y)}~dY_{inv}
\end{equation}
is the $\Gamma(\rho_r\otimes\det^s)$-eigenvalue of $V_k(z)$, which in particular  is independent of $\nu$.
\begin{prop}\label{SO(2)-Gamma-eigenvalues}
For $k=0,1,\dots,r$ we have the functional equation
\begin{equation*}
 \Gamma(r,k,s)\:=\:\Gamma(r,r-k,s)\:.
\end{equation*}
 For $k=0,1,\dots,\lfloor \frac{r}{2}\rfloor$ the function $\Gamma(r,k,s)$ is explicitly given by
 \begin{equation*}
  \Gamma(r,k,s)\:=\:\Gamma_2(s)\sum_{\mu=0}^{\lfloor\frac{r}{2}\rfloor}   \frac{a_{2\mu-1,\mu}}{2^\mu}\sum_{j=0}^{k}\binom{k}{j}\binom{r-2k}{2(\mu-j)}(-1)^{\mu-j}
  \prod_{l=0}^{r-\mu-1}(s+l)\:.
 \end{equation*}
 With respect to the $\SO(2)$-weight decomposition, the operator $\Gamma(\rho_r\otimes\det^s)$ is given by the diagonal matrix
 \begin{equation*}
  \Gamma(\rho_r\otimes\det\nolimits^s)\:=\:\mathop{diag}\bigl(\Gamma(r,0,s),\Gamma(r,1,s),\dots,\Gamma(r,\lfloor\frac{r}{2}\rfloor,s),\dots,\Gamma(r,1,s),\Gamma(r,0,s)\bigr)\:.
 \end{equation*}
In particular, $\Gamma(\rho_r\otimes\det^s)$ is divisible by $\Gamma_2(s)\prod_{l=0}^{\lfloor\frac{r}{2}\rfloor-1}(s+l)$. Apart from its finite set of zeros and its set of poles 
which is contained in that of $\Gamma_2(s)$, the operator  $\Gamma(\rho_r\otimes\det^s)$ is invertible for $\re s>\frac{1}{2}$.
\end{prop}
\begin{proof}[Proof of proposition~\ref{SO(2)-Gamma-eigenvalues}]
We determine $\Gamma(r,k,s)$ by choosing $\nu=0$ in (\ref{gamma-r-k-expression}). 
For integers $a,b\geq 0$ 
 \begin{equation*}
  (Y_{11}^2+Y_{12}^2)^{a}(Y_{11}\pm iY_{12})^{b}\:=\:\sum_{j=0}^a\sum_{l=0}^b\binom{a}{j}\binom{b}{l}(\pm i)^lY_{11}^{2(a-j)+b-l}Y_{12}^{2j+l}\:,
 \end{equation*}
 so by lemma~\ref{integral_mit_Y_eintraegen} only the summands with even $Y_{12}$-exponents contribute to the integral
  \begin{align*}
 &\int_{\mathcal Y}(Y_{11}^2+Y_{12}^2)^{a}(Y_{11}\pm iY_{12})^{b}\det(Y)^se^{-\trace(Y)}~dY_{inv}\\
 &=\Gamma_2(s)\sum_{j=0}^a\sum_{l=0}^{\lfloor\frac{b}{2}\rfloor}\binom{a}{j}\binom{b}{2l}(-1)^l\frac{a_{2(j+l)-1,j+l}}{2^{j+l}}\prod_{\mu=0}^{2a+b-(j+l)-1}(s+\mu)\:.
\end{align*}
Notice that the integral is independent of the sign in $(Y_{11}\pm iY_{12})^{b}$.
Accordingly $\Gamma(r,k,s)=\Gamma(r,r-k,s)$, and we may restrict  to the case $k\leq r-k$,
and apply the above formula with $a=k$ and $b=r-2k$.
\end{proof}
In particular, in case $k=0$
\begin{equation}\label{gamma-r-0-eq-1}
 \Gamma(r,0,s)\:=\:\Gamma_2(s)\sum_{\mu=0}^{\lfloor\frac{r}{2}\rfloor}\binom{r}{2\mu}(-1)^\mu\frac{a_{2\mu-1,\mu}}{2^{\mu}}\prod_{l=0}^{r-\mu-1}(s+l)\:.
\end{equation}
On the other hand,  we recall the formula valid for all $\nu$
\begin{equation*}
 \Gamma(r,0,s)\:=\:\int_{\mathcal Y}(Y_{11}-iY_{12})^{r-\nu}(Y_{11}+iY_{12})^\nu\det(Y)^se^{-\trace(Y)}~dY_{inv}\:.
\end{equation*}
Because
\begin{equation*}
 (Y_{11}+Y_{22})^r\:=\:\sum_{\nu=0}^r\binom{r}{j}(Y_{11}-iY_{12})^{r-\nu}(Y_{22}+iY_{12})^\nu\:,
\end{equation*}
we obtain
\begin{equation*}
 \int_{\mathcal Y}(Y_{11}+Y_{22})^r\det(Y)^se^{-\trace(Y)}~dY_{inv}\:=\: 2^r\cdot \Gamma(r,0,s)\:,
\end{equation*}
which implies
\begin{equation*}
 \Gamma(r,0,s)\:=\: \frac{\Gamma_2(s)}{2^r}\sum_{j=0}^r\binom{r}{j}\sum_{\mu=0}^{\min\{j,r-j\}}\binom{r-j}{\mu}\binom{j}{\mu}(-1)^\mu\mu!\prod_{l=0}^{r-\mu-1}(s+l)\:,
\end{equation*}
or equivalently
\begin{equation}\label{gamma-r-0-eq-2}
 \Gamma(r,0,s)\:=\: \Gamma_2(s)\sum_{\mu=0}^{\lfloor\frac{r}{2}\rfloor}
 (-1)^\mu\frac{\mu!}{2^r}\sum_{j=\mu}^{r-\mu}
 \binom{r}{j}\binom{r-j}{\mu}\binom{j}{\mu}\prod_{l=0}^{r-\mu-1}(s+l)\:.
\end{equation}
Noticing that the polynomials $\prod_{l=0}^{r-\mu-1}(s+l)$ for $\mu=0,\dots,\lfloor\frac{r}{2}\rfloor$ are linearly independent, we obtain 
by comparing the coefficients of (\ref{gamma-r-0-eq-1}) 
and (\ref{gamma-r-0-eq-2})
\begin{equation*}
 \frac{\mu!}{2^r}\sum_{j=\mu}^{r-\mu}
 \binom{r}{j}\binom{r-j}{\mu}\binom{j}{\mu}\:=\:\binom{r}{2\mu}\frac{a_{2\mu-1,\mu}}{2^\mu}\:,
\end{equation*}
which is easily simplified to the identity of proposition~\ref{a_2n-1,n-coefficient} (a) below.

\begin{prop}\label{a_2n-1,n-coefficient}
The triangle numbers defined in Proposition~\ref{triangle_numbers} take the following special values.
\begin{itemize}
 \item [(a)]
 For all $\mu=0,1,2,\dots$,
 \begin{equation*}
 a_{2\mu-1,\mu}\:=\:\frac{(2\mu)!}{2^\mu \mu!}\:=\:(2\mu-1)!!\:.
\end{equation*}
\item [(b)]
For all $\mu=1,2,3,\dots$,
\begin{equation*}
 a_{2\mu-1,\mu-1}\:=\:\mu\cdot (2\mu-1)!!\:.
\end{equation*}
\item [(c)]
For all $\mu=1,2,3,\dots$,
\begin{equation*}
 a_{2\mu,\mu-1}\:=\:\frac{\mu}{3} (2\mu+1)!!\:.
\end{equation*}
\end{itemize}
\end{prop}
\begin{proof}[Proof of Proposition~\ref{a_2n-1,n-coefficient}]
By the defining recursion formula we obtain
 \begin{equation*}
  a_{2\mu,\mu}\:=\:2a_{2\mu-1,\mu-1}+a_{2\mu-1,\mu}\:.
 \end{equation*}
Because part~(a) has already been verified for all $\nu$, we obtain part~(b) by using proposition~\ref{triangle_numbers}~(iv)
\begin{equation*}
 a_{2\mu-1,\mu-1}\:=\:\frac{1}{2}\bigr((2\mu+1)!!-(2\mu-1)!!\bigr)\:=\: \mu\cdot (2\mu-1)!!\:.
\end{equation*}
By recursion $a_{2\mu+1,\mu}=3a_{2\mu,\mu-1}+a_{2\mu,\mu}$,
 and applying (a) and (b), we obtain part~(c)
\begin{equation*}
 a_{2\mu,\mu-1}\:=\:\frac{1}{3}\bigl((\mu+1)(2\mu+1)!!-(2\mu+1)!!\bigr)\:=\:\frac{\mu}{3}(2\mu+1)!!\:.\qedhere
\end{equation*}
\end{proof}
\begin{exmp}[Symmetric representation]
For $r=2$ the representation $\rho_2$ is isomorphic to the symmetric representation. 
%
%
In terms of the basis of eigenvectors $V_k(z)$, $k=-2,0,2$, for  $\SO(2)$,
the $\Gamma$-integral is given by the matrix
\begin{equation*}
 \Gamma(\rho_2\otimes\det\nolimits^s)\:=\:s\Gamma_2(s)\begin{pmatrix}s+\frac{1}{2}&&\\&s+\frac{3}{2}&\\&&s+\frac{1}{2}                             
                                          \end{pmatrix}\:.
\end{equation*}
Equivalently, on the space of symmetric matrices $X=\left(\begin{smallmatrix}X_1&X_{12}\\X_{12}&X_2\end{smallmatrix}\right)$,
\begin{align*}
  \Gamma((\mathop{Sym}\otimes\det\nolimits^s)(X))&=\int_{Y>0}YXY\det(Y)^se^{-\trace(Y)}dY_{inv}\\
  &=
  s(s+1)\Gamma_2(s)\cdot X+\frac{s}{2}\Gamma_2(s)\cdot\tilde X\:,
 \end{align*}
 where $\tilde X=\left(\begin{smallmatrix}X_2&-X_{12}\\-X_{12}&X_1\end{smallmatrix}\right)$ is the adjunct matrix for $X$.
  In particular, this example shows that $\Gamma(\rho)$ is not a scalar operator in general.
\end{exmp}


\subsection{Weyl's character formula}\label{sec-weyl}

Lemma~\ref{orthogonale_invarianz} suggests the following integral transformation.
For the   diagonal torus $T$ of $\GL(m,\RR)$ let 
\begin{equation*}
 T_{>0}\:=\: \{t=\mathop{diag}(t_1,\dots, t_m)\in T \mid t_1>t_2>\dots>t_m\}\:.
\end{equation*}
Denote by $\P_m$ the set of positive definite $(m,m)$-matrices, $\P_m\subset \mathop{Sym}\nolimits^2(\RR^m)$.
Let $K=\SO(m)$ with unit element $E$. There is an injective map 
\begin{align*}
& T_{>0}\times K\:\to\:\P_m\:,\quad   (t,k)\:\mapsto\: ktk'\:=\:kt k^{-1}\:=\:Y\:,
\end{align*}
which has open and dense image.
For the pullback $\phi^\ast$ we find
\begin{equation*}
 \phi^\ast(dY)\:(t,E)\:=\:(dX'\cdot t+t\cdot dX)+dt\:,
\end{equation*}
where
\begin{equation*}
 dX\:=\:\begin{pmatrix}0&dx_{12}&\dots&dx_{1m}\\
         -dx_{12}&\ddots&&\vdots\\
         \vdots&\dots&   \dots&0     \end{pmatrix}\:.
\end{equation*}
So $-dX\cdot t +t\cdot dX$ equals
\begin{equation*}
 \begin{pmatrix}
                             0&(t_1-t_2)dx_{12}&\dots&(t_1-t_m)dx_{1m}\\
                             (t_1-t_2)dx_{12}&\ddots&\ddots&\vdots \\
                             \vdots&\dots&(t_{m-1}-t_m)dx_{m-1,m}&0
                            \end{pmatrix}\:.
\end{equation*}
Accordingly, the pullback $\phi^\ast(\det(Y)^{-\frac{m+1}{2}}\prod_{i\leq j}dY_{ij})$ at $(t,E)$ of the invariant measure $dY_{inv}$
on $\P_m$ is given by
\begin{equation*}
 \pm\det(t)^{-\frac{m+1}{2}}\cdot \prod_{i<j}(t_i-t_j)\bigwedge\nolimits_{i<j}dx_{ij}\wedge (dt_1\wedge\dots\wedge dt_m)\:.
\end{equation*}
Since $\phi^\ast(dY_{inv})$ is $K$-invariant, we obtain
\begin{equation}\label{weylintegral}
\Gamma(\rho)\:=\: \pm\int_K\int_{T_{>0}}\det(t)^{-\frac{m+1}{2}}\prod_{i<j}(t_i-t_j)\rho(ktk')e^{-\trace(t)}~dt~dk\:.
\end{equation}
We double check this formula by testing it for $\rho=\det\nolimits^k$ in case $m=2$, where
\begin{equation*}
 \Gamma(\det\nolimits^k)\:=\:\Gamma_2(k)\:=\:\sqrt\pi\cdot\Gamma(k)\Gamma(k-\frac{1}{2})\:.
\end{equation*}
This must equal up to a constant  depending on the normalization of measures and their orientation
\begin{equation*}
 \int_{t_2>t_1>0}(t_1t_2)^{k-\frac{3}{2}}(t_1-t_2)e^{-t_1-t_2}~dt_1~dt_2\:,
 \end{equation*}
which equals
\begin{equation*}
 \int_0^\infty t_2^{k-\frac{3}{2}}e^{-t_2}~dt_2\cdot \int_0^\infty t_2^{k-\frac{1}{2}}e^{-t_1}~dt_1\:-\:
 2\int_0^\infty t_2^{k-\frac{3}{2}}e^{-t_2}\int_{t_2}^\infty t_1^{k-\frac{1}{2}}e^{-t_1}~dt_1~dt_2\:.
\end{equation*}
For the last integral we first notice that by partial integration
\begin{equation*}
 \int_{t_2}^\infty t_1^{k-\frac{1}{2}}e^{-t_1}~dt_1\:=\: t_2^{k-\frac{1}{2}}e^{-t_2}\:+\:(k-\frac{1}{2})\int_{t_2}^\infty t_1^{k-\frac{3}{2}}e^{-t_1}~dt_1\:.
\end{equation*}
Let $\phi(t)$ be an antiderivative of $-t^{k-\frac{3}{2}}e^{-t}$, in particular
\begin{equation*}
 \phi(t_2)\:=\:\int_{t_2}^\infty t_1^{k-\frac{3}{2}}e^{-t_1}~dt_1\:.
\end{equation*}
Accordingly,
\begin{eqnarray*}
\int_{0}^\infty \phi'(t_2)\phi(t_2)~dt_2
 &=& \int_0^\infty t_2^{k-\frac{3}{2}}e^{-t_2}\int_{t_2}^\infty t_1^{k-\frac{3}{2}}e^{-t_1}~dt_1~dt_2\\
  &=& \phi(t_2)^2\mid_0^\infty \:-\:\int_0^\infty \phi(t_2)\phi'(t_2)~dt_2\:,
\end{eqnarray*}
i.e.
\begin{equation*}
 -2\int_0^\infty t_2^{k-\frac{3}{2}}e^{-t_2}\int_{t_2}^\infty t_1^{k-\frac{3}{2}}e^{-t_1}~dt_1~dt_2\:=\: -\Gamma(k-\frac{1}{2})^2\:.
\end{equation*}
So we obtain $\int_{t_2>t_1>0}(t_1t_2)^{k-\frac{3}{2}}(t_1-t_2)e^{-t_1-t_2}~dt_1~dt_2$ to equal
\begin{equation*}
 \Gamma(k+\frac{1}{2})\Gamma(k-\frac{1}{2})-\frac{1}{2^{2k-2}}\Gamma(2k-1)\:-\:(k-\frac{1}{2})\Gamma(k-\frac{1}{2})^2\:,
\end{equation*}
which simplifies to $-\left(\frac{1}{2}\right)^{2k-2}\Gamma(2k-1)$. Using Legendre's relation,
\begin{equation*}
 \frac{\sqrt\pi}{2^{z-1}}\Gamma(z)\:=\:\Gamma(\frac{z}{2})\Gamma(\frac{z+1}{2})
\end{equation*}
we conclude
\begin{equation*}
 \int_{t_2>t_1>0}(t_1t_2)^{k-\frac{3}{2}}(t_1-t_2)e^{-t_1-t_2}~dt_1~dt_2\:=\: -\frac{1}{\sqrt \pi}\cdot\Gamma(k)\Gamma(k-\frac{1}{2})\:.
 \end{equation*}
Thus indeed,
\begin{equation*}
 \Gamma_2(\det\nolimits^k)\:=\: (-1)\cdot \vol(K)\cdot  \int_{t_2>t_1>0}(t_1t_2)^{k-\frac{3}{2}}(t_1-t_2)e^{-t_1-t_2}~dt_1~dt_2
\end{equation*}
with the volume of $K$ normalized by $\vol(K)=\pi$.
\bigskip

We use (\ref{weylintegral}) to compute  the trace 
\begin{equation*}
 \trace(\Gamma(\rho))\:=\: \pm\vol(K)\int_{T_{>0}}\det(t)^{-\frac{m+1}{2}}\prod_{i<j}(t_i-t_j)\trace(\rho(t))e^{-\trace(t)}~dt\:.
\end{equation*}
On the other hand, we can use Weyl's character formula
\begin{equation*}
 \chi_{\rho}\:=\:\frac{\sum_{w\in S_m}\sign(w)e^{w(\lambda+\delta)}}{\prod_{\alpha\in\Phi^+}(e^{\frac{\alpha}{2}}-e^{-\frac{\alpha}{2}})}\:=\:
\frac{\sum_{w\in S_m}\sign(w)e^{w(\lambda+\delta)-\delta}}{\prod_{\alpha\in\Phi^+}(1-e^{-\alpha})}
 \:,
\end{equation*}
where $\delta$ is half the sum of positive roots of $\SO(m)$, to compute $\trace(\rho(t))=\chi_\rho(t)$.
In case the rank $m=2m'\geq 4$ is even, a system $\Phi^+$ of positive roots is given by $e_i-e_j$ and $e_i+e_j$ for $1\leq i<j\leq m$, so  $\delta=\sum_{i}(m-i)e_i$ for $m=2m'$. We obtain
\begin{equation*}
 \trace(\rho(t))\:=\:\chi_\rho(t)\:=\:\frac{\sum_{w\in S_m}\sign(w)t^{w(\lambda+\delta)-\delta}}{\prod_{i<j}(1-\frac{t_j}{t_i})}\:,
\end{equation*}
where for a vector $v=(v_1,\dots,v_m)$ we write $t^v=t_1^{v_1}\cdots t_m^{v_m}$. Thus, in the even rank case we obtain 
 \begin{equation*}
 \trace(\Gamma(\rho))\:=\: \pm\vol(K)\sum_{w\in S_m}\sign(w)\int_{T_{>0}}t^{w(\lambda+\delta)-\frac{m+1}{2}}e^{-\trace(t)}~dt\:.
\end{equation*}
In case the rank $m=2m'+1\geq 3$ is odd, there are the additional positive roots~$e_i$, $i=1,\dots, m$, so 
 $ \delta= \sum_{i}(m+\frac{1}{2}-i)e_i$ and 
\begin{equation*}
\chi_\rho(t)\:=\:\frac{\sum_{w\in S_m}\sign(w)t^{w(\lambda+\delta)-\delta}}{\prod_{i<j}(1-\frac{t_j}{t_i})\cdot \prod_i (1-t_i^{-1})}\:.
\end{equation*}
Thus in the case of odd rank
\begin{equation*}
 \trace(\Gamma(\rho))\:=\: \pm\vol(K)\sum_{w\in S_m}\sign(w)\int_{T_{>0}}\frac{t^{w(\lambda+\delta)-\frac{m}{2}}}{\prod_i(t_i-1)}e^{-\trace(t)}~dt\:.
\end{equation*}



\begin{thebibliography}{10}
\bibitem{panchishkin}
M.~Courtieu, A.~Panchishkin:
\textit{Non-archimedean L-functions and arithmetical Siegel modular forms,}
Lecture Notes in Mathematics 1471, second augmented edition,
Springer (2004),
Heidelberg u.a.

 \bibitem{freitag}
E.~Freitag:
\textit{Siegelsche Modulformen,}
Grundlehren der mathematischen Wissenschaften 254 (1983), Springer

\bibitem{freitag2}
E.~Freitag:
\textit{Singular modular forms and theta relations,} Lecture notes in mathematics 1487, Springer (1991).

\bibitem{godement}
R.~Godement:
\textit{Fonctions holomorphes de carre\'e sommable dans le demi-plan de Siegel,} Sem. H.~Cartan 6, E.~N.~S. (1957/58), 1-22.

\bibitem{gross-zagier}
B.~H.~Gross, D.~B.~Zagier: 
\textit{Heegner points and derivatives of L-series,} 
Invent. Math. 84 (1986), no. 2, 225-320.


\bibitem{holproj}
K.~Maurischat:
\textit{On holomorphic projection for symplectic groups,}  J. Number Theory, Vol. 182 (2018), 131-178.


\bibitem{scalar-weight}
K.~Maurischat:
\textit{ Sturm's operator for scalar weight in arbitrary genus,} Int. J. Number Theory, Vol. 13, No. 10 (2017), pp. 2677-2686. 

\bibitem{phantom}
K.~Maurischat, R.~Weissauer:
{\it Phantom holomorphic projections arising from Sturm's formula,} The Ramanujan J., 47(1) (2018), 21-46.

\bibitem{sturm1}
Sturm, J.:
\textit{Projections of $C^\infty$ automorphic forms,}
Bull. Amer. Math. Soc. 2 (1980), 435-439.

\bibitem{sturm2}
Sturm, J.:
\textit{The critical values of Zeta-functions associated to the symplcetic group,}
Duke Math. J. 48 (1981), 327-350.


\end{thebibliography}
\end{document}